\documentclass[leqno]{article}
\usepackage{amsmath}
\usepackage{amsthm}
\usepackage{amsfonts}
\usepackage{mathtools}
\usepackage{enumerate}
\usepackage[hmargin=3cm,vmargin=3.5cm]{geometry}
\usepackage{amssymb}

\newcommand{\Rho}{\mathrm{P}}
\makeatletter
\renewcommand*\@fnsymbol[1]{\the#1} 
\makeatother
\newtheorem{theorem}{Theorem}
\newtheorem{lemma}[theorem]{Lemma}
\newtheorem{corollary}[theorem]{Corollary}
\theoremstyle{definition} \newtheorem*{definition}{Definition}
\theoremstyle{definition} \newtheorem*{notation}{Notation}
\theoremstyle{definition} \newtheorem*{note}{Note}
\newenvironment{atheorem}[2][Theorem]{\begin{trivlist}
\item[\hskip \labelsep {\bfseries #1}\hskip \labelsep {\bfseries #2}]}{\end{trivlist}}
\def\noqed{\renewcommand{\qedsymbol}{}}
\numberwithin{equation}{section}

\begin{document}

\title{On Commutation Semigroups of Dihedral Groups}
\author{Darien DeWolf\thanks{Mathematics Department, Dalhousie University, Halifax, Nova Scotia , Canada, B3H 4R2} \and Charles Edmunds\thanks{Mathematics Department, Mount Saint Vincent University, Halifax, Nova Scotia, Canada, B3M 2J6} \and Christopher Levy\thanks{Mathematics Department, Dalhousie University, Halifax, Nova Scotia , Canada, B3H 4R2}}
\date{}
\maketitle
\begin{center}
\emph{In Memory of Narain Gupta}
\end{center}

\section{Introduction} 
For any group $ G $ with $ g\in G $, the right and left commutation mappings associated with $ g $ are the mappings $ \rho (g){\text{ and }}\lambda (g) $ from $ G $ to $ G $ defined as
\[
(x)\rho (g) = [x,g] \mbox{ and } (x)\lambda (g) = [g,x],
\]
where the commutator of $ g $ and $ h $ is defined as $ [g,h] = {g^{ - 1}}{h^{ - 1}}gh $. The set $ \mathcal{M}(G) $ of all mappings from $ G $ to $ G $ forms a semigroup under composition of mappings.  The \emph{right commutation semigroup of} $ G $, $ {\rm P}(G) $,  is the subsemigroup of $ \mathcal{M}(G) $ generated by the set of $ \rho $\emph{-maps}, $ {{\rm P}_1}(G) = \left\{ {\rho (g):g \in G} \right\} $, and \emph{the left commutation semigroup of} $ G $, $ \Lambda (G) $, is the subsemigroup of $ \mathcal{M}(G) $ generated by the set of $ \lambda $\emph{-maps},  $ {\Lambda _1}(G) = \left\{ {\lambda (g):g \in G} \right\} $. If $ G $ is abelian the commutation semigroups are trivial semigroups consisting of one mapping sending each element of $ G $ to the identity element. We will study $ {\rm P}(G) $ and $ \Lambda (G) $ only when $ G $ is non-abelian. 

In this paper we will discuss the commutation semigroups of dihedral groups. The dihedral group of order $ 2m $ has presentation
\[
{D_m} = \left\langle {a,b;{a^m} = 1,{b^2} = 1,{a^b} = {a^{ - 1}}} \right\rangle ,
\]
where the conjugate of $ a $ by $ b $ is denoted $ {a^b} = {b^{ - 1}}ab $. Each element of $ {D_m} $ can be written uniquely in the form $ {a^i}{b^j} $ with $ i \in {\mathbb{Z}_m} $ and $ j \in {\mathbb{Z}_2} $. Since $ {D_3} $ is the smallest non-abelian dihedral group, we will assume, henceforth, that $ m \geq 3 $. Our primary goal is to develop explicit formulas for the orders of $ {\rm P}(G) $ and $ \Lambda (G) $.

In the mid-1960s B.H. Neumann pointed out to N.D. Gupta (oral communication) that $ \left| {{\rm P}({D_3})} \right| = 6 $ but $ \left| {\Lambda ({D_3})} \right| = 9 $. One might have thought, at first glance, that the left and right commutator semigroups would be isomorphic. However, the smallest nonabelian group yields a counterexample and this raised the question of how these two semigroups are related.  In \cite{gupta66}, Gupta characterized those dihedral groups for which $ {\rm P}(G) $ and $ \Lambda (G) $ are isomorphic. He then went on to study the question of isomorphism for nilpotent groups, finding that, for groups of class 2, 3, and 4, one has $ {\rm P}(G) = \Lambda (G) $, $ {\rm P}(G) \cong \Lambda (G) $, and $ \left| {{\rm P}(G)} \right| = \left| {\Lambda (G)} \right| $, respectively. He then gave an example of a class 5 group for which the commutation semigroups are not isomorphic. The question of whether $ {\rm P}(G) \cong \Lambda (G) $ for class 4 groups is still open.

In the mid-1960’s the Neumanns were making significant contributions to variety theory and it is reasonable to suppose that the purpose of the study of commutation semigroups was to further the understanding of the varieties of the groups with which they are associated. By interpreting the group multiplication as (noncommutative) addition, Gupta \cite{gupta67} also studied these semigroups as the multiplicative structures of commutation near rings. 

In 1970 James Countryman \cite{countryman70} wrote his Ph.D. thesis at the University of Notre Dame on the commutation semigroups of $pq$ groups for $p$ and $q$ distinct primes with $ p < q $. Each nonabelian $pq$ group is a split extension of a cyclic group of order $q$ by a cyclic group of order $p$. Among his results were the following:

\begin{atheorem}{C1}
If $ G $ is a $ pq $ group the following statements are equivalent:
\begin{enumerate}[(a)]
	\item $ {\rm P}(G) = \Lambda (G), $ 
	\item $ {\rm P}(G) \cong \Lambda (G), $ 
	\item $ \left| {{\rm P}(G)} \right| = \left| {\Lambda (G)} \right|. $ 
\end{enumerate}
\end{atheorem}

\begin{atheorem}{C2}
If $ {G_1}{\text{ and }}{G_2} $ are $ pq $ groups, then $ {\rm P}({G_1}) \cong {\rm P}({G_2}) $ implies $ {G_1} \cong {G_2} $.
\end{atheorem}

\begin{atheorem}{C3}
If $ G $ is a $ 2q $ group, then $ {\rm P}(G) \subseteq \Lambda (G) $ or $ \Lambda (G) \subseteq {\rm P}(G) $, or both.
\end{atheorem}

Since the non-abelian $2q$ groups are among the dihedral groups, one might conjecture that these results hold for all dihedral groups. In the case of Theorem C1 it is clear, for any group, that $ (a) \Rightarrow (b) \Rightarrow (c) $. We will show that, for dihedral groups in general, $ (c)\not  \Rightarrow (a) $, $ (b)\not  \Rightarrow (a) $, and $ (c)\not  \Rightarrow (b) $ On the positive side, we will derive Theorem \ref{t23}, a left commutation semigroup version of Theorem C2. We note as well that Theorem C2, our Theorem \ref{t23}, and Theorem C3 do not hold for dihedral groups in general.

The difficulty in identifying the elements of $ {\rm P}(G) $ and $ \Lambda (G) $ is that, although the generating sets $ {{\rm P}_1}(G){\text{ and }}{\Lambda _1}(G) $ are clearly defined, these generators must be multiplied together (composed) repeatedly to form the full set of mappings in $ {\rm P}(G) $ and $ \Lambda (G) $. We will develop a method which gives  more control over this process than has been possible previously. The method evolved as joint work while the second author supervised the Honours Theses of the first and third authors at Mount Saint Vincent University. Their work was complementary, each author giving formulas for the orders of $ {\rm P}({D_m}){\text{ and }}\Lambda ({D_m}) $; Levy for $ m $ odd or $m$ a power of 2 and DeWolf for $m$ even. 

\begin{atheorem}{L1 (Levy \cite{levy08})}
If $m$ is odd, then
\begin{enumerate}[(i)]
	\item $ \left| {{\rm P}({D_m})} \right| = m(in{d_R}(m) + 1) $, where $ in{d_R}(m) = \min \left\{ {i \in {\mathbb{Z}^ + }:{{( - 2)}^i} \equiv 1\;(\bmod m)} \right\}, $
	\item $ \left| {\Lambda ({D_m})} \right| = m(in{d_L}(m) + 1) $, where $ in{d_L}(m) = \min \left\{ {i \in {\mathbb{Z}^ + }:{2^i} \equiv 1\;(\bmod m)} \right\} .$
\end{enumerate}
\end{atheorem}

\begin{atheorem}{L2 (Levy \cite{levy08}) }
If $ m = {2^\ell } $, then $ \left| {{\rm P}({D_m})} \right| = \left| {\Lambda ({D_m})} \right| = {2^\ell } + {2^{\ell  - 1}} - 2 $.
\end{atheorem}

\begin{atheorem}{D (DeWolf \cite{dewolf12})}
If $ m = {2^\ell }n $ with $ {\text{n odd}} $ and $ \ell ,n \in {\mathbb{Z}^ + } $,then
\begin{enumerate}[(i)]
	\item $ \left| {{\rm P}({D_m})} \right| = n({2^\ell } + {2^{\ell  - 1}} - 2 + {m_R} - \ell ) $, 
where \\ $ {m_R} = \min \left\{ {i > \ell :{{( - 2)}^i} \equiv {{( - 2)}^\ell }\;(\bmod m)} \right\}\;{\text{if n}} > 1 $, and  $ {m_R} = \ell {\text{ if n}} = 1 ,$
	\item $ \left| {\Lambda ({D_m})} \right| = n({2^\ell } + {2^{\ell  - 1}} - 2 + {m_L} - \ell ) $, where \\ $ {m_L} = \min \left\{ {i > \ell :{2^i} \equiv {2^\ell }\;(\bmod m)} \right\}\;{\text{if n}} > 1 $, and
              $ {m_L}{\kern 1pt}  = \ell {\text{ if n}} = 1 $.
\end{enumerate}
\end{atheorem}

These formulas were used to generate the table at the end of this paper giving orders of the left and right commutation semigroups for $ {D_m} $  $ \left( {3 \leq m \leq 101} \right) $. The formulas in Theorems L2 and D do have some  similarities, but it is unsatisfying to have such different looking formulas for the various types of dihedral groups. It is possible, however, to give one formula each for the orders of $ {\rm P}({D_m}){\text{ and }}\Lambda ({D_m}) $. This formula involves the orders of the terms of the upper central series of $ {D_m} $, thereby hinting at the value of a more group theoretical approach to the questions surrounding commutation semigroups. We will prove the following formulas and derive the results of Levy and DeWolf from them.

\begin{atheorem}{21}
If $ m = {2^\ell }n > 3 $ with $ n $ odd,
\begin{enumerate}[(i)]
	\item $
	\left| {{\rm P}({D_m})} \right| = m\left( {\dfrac{1}{{\left| {{Z_1}} \right|}} + \sum\limits_{i = 1}^{t - 1} {\dfrac{1}{{\left| {{Z_i}} \right|}}} } \right)\mbox{, where }    
	t = \left\{ {
	\begin{array}{*{20}{c}}
 		{1 + or{d_m}( - 2){\text{ for }}\ell  = 0,\,n > 1} \\ 
  		{\ell  + pe{r_m}( - 2){\text{ for }}\ell  > 0,\,n > 1} \\ 
  		{\ell {\text{                   for }}\ell  > 0,\,n = 1} 
	\end{array}} \right.  ,$
	\item $ 
	 \left| {\Lambda ({D_m})} \right| = m\left( {\dfrac{1}{{\left| {{Z_1}} \right|}} + \sum\limits_{i = 1}^{t' - 1} {\dfrac{1}{{\left| {{Z_i}} \right|}}} } \right)\mbox{, where }     
	t' = \left\{ {
	 \begin{array}{*{20}{c}}
  		{1 + or{d_m}(2){\text{ for }}\ell  = 0,\,n > 1} \\ 
  		{\ell  + pe{r_m}(2){\text{ for }}\ell  > 0,\,n > 1} \\ 
 		{\ell {\text{                   for }}\ell  > 0,\,n = 1} 
	\end{array}} \right. , $
	
\end{enumerate}
where $ {Z_i} $ is the $ {i^{th}} $ centre of $ {D_m} $, that is, the $ {i^{th}} $ term of the upper central series.
\end{atheorem}

In Section 2 we develop the initial ideas and notation used throughout the paper. Section 3 introduces containers, a method of binding the commutation mappings together into natural units. These containers actually form a quotient semigroup of a natural subsemigroup of $ \mathcal{M}(G) $ implicit in Gupta \cite{gupta66}. Rather than working with individual mappings, we will show how these containers can be multiplied repeatedly to generate $ {\rm P}({D_m}){\text{ and }}\Lambda ({D_m}) $. In Section 4 we calculate the cardinality of each container in terms of the upper central series. In Section 5 we state and prove our main theorem and give two applications. In Section 6 we show how to use our methods to derive the formulas of Levy and DeWolf and discuss possible generalizations of Countryman's theorems to all dihedral groups. It is hoped that this will show the merit of our approach to the reader.

\section{Preliminaries}
The fundamental concepts developed in this section are parallel to those developed by N.D. Gupta in \cite{gupta66}. Many of these ideas are explicitly or implicitly his.

The following commutator identities are easily verified by expansion.
\begin{atheorem}{CI}
If G is any group and $ x,y,z \in G $ then
\begin{enumerate}[(i)]
	\item $ {x^y} = x[x,y] $, 	
	\item $ [y,x] = {[x,y]^{ - 1}} $,
	\item $ [xy,z] = {[x,z]^y}[y,z] = [x,z][x,z,y][y,z] $, 
	\item $ [x,yz] = [x,z]{[x,y]^z} = [x,z][x,y][x,y,z] $,	
	\item $ [{x^{ - 1}},y] = {[x,y]^{ - {x^{ - 1}}}} $.\qed
\end{enumerate} 
\end{atheorem}
\begin{note}
For any group $G$ and $a,b\in G,$ we denote both $(a^{-1})^b$ and $(a^{-1})^b$ by $a^{-b}.$ This is unambiguous since $(a^{-1})^b = b^{-1}a^{-1}b = (b^{-1}ab)^{-1} = (a^b)^{-1}.$
\end{note}
\begin{notation}
For each $ s \geq 0 $ let $ {\alpha _s} = {( - 1)^s} $ and $ {\beta _s} = {( - 1)^s} - 1 $. Since the values of $ {\alpha _s}{\text{ and }}{\beta _s} $ are unique up to parity, it will cause no ambiguity to view s as an element of $ {\mathbb{Z}_2} $.
\end{notation}

We begin by calculating an explicit formula for each $ \rho $- and $ \lambda $-map.

\begin{lemma}
\label{l1}
Let $ {D_m} $ be the dihedral group with presentation as above. For each $ i,r \in {\mathbb{Z}_m} $ and $ j,s \in {\mathbb{Z}_2}:$
\[
	({a^i}{b^j})\rho ({a^r}{b^s}) = {a^{{N_\rho }}} \mbox{ and } ({a^i}{b^j})\lambda ({a^r}{b^s}) = {a^{{N_\lambda }}},
\]
where $ {N_\rho } \equiv i{\alpha _j}{\beta _s} - r{\alpha _s}{\beta _j} \equiv ( - 2){\alpha _{js}}(is - jr)(\bmod m) $ and $ {N_\lambda } \equiv  - {N_\rho } \equiv 2{\alpha _{js}}(is - jr)(\bmod m) $.
\end{lemma}

\begin{proof}
Applying CI \emph{(ii)}, CI \emph{(iii)}, and CI \emph{(iv)} we have,
\begin{align*}
({a^i}{b^j})\rho ({a^r}{b^s}) &= [{a^i}{b^j},{a^r}{b^s}] = {[{a^i},{a^r}{b^s}]^{{b^j}}}[{b^j},{a^r}{b^s}] \\
	&= \left( {{{[{a^i},{b^s}]}^{{b^j}}}{{[{a^i},{a^r}]}^{{b^s}{b^j}}}} \right)\left( {[{b^j},{b^s}]{{[{b^j},{a^r}]}^{{b^s}}}} \right) \\
	&= {[{a^i},{b^s}]^{{b^j}}}{[{b^j},{a^r}]^{{b^s}}} 
	  = {[{a^i},{b^s}]^{{{( - 1)}^j}}}\left( {{{[{a^r},{b^j}]}^{ - {{( - 1)}^s}}}} \right) \\ 
  	&= {\left( {{a^{ - i}}{{({a^i})}^{{b^s}}}} \right)^{{{( - 1)}^j}}}{\left( {{a^{ - r}}{{({a^r})}^{{b^j}}}} \right)^{ - {{( - 1)}^s}}} \\
	&= {\left( {{a^{ - i}}{{({a^i})}^{{{( - 1)}^s}}}} \right)^{{{( - 1)}^j}}}{\left( {{a^{ - r}}{{({a^r})}^{{{( - 1)}^j}}}} \right)^{ - {{( - 1)}^s}}} \\ 
	&= {a^{i( - 1 + {{( - 1)}^s}){{( - 1)}^j} - r( - 1 + {{( - 1)}^j}){{( - 1)}^s}}} = {a^{{N_\rho }}},
\end{align*}
where $ {N_\rho } \equiv i{\beta _s}{\alpha _j} - r{\beta _j}{\alpha _s} $, as required. Checking each of the four cases for $ (j,s) \in \mathbb{Z}_2^2 $, we also see that $ {N_\rho } \equiv ( - 2){\alpha _{js}}(is - jr)(\bmod m) $. The first part of the lemma can be used to prove the second part.
\[
	({a^i}{b^j})\lambda ({a^r}{b^s}) = [{a^r}{b^s},{a^i}{b^j}] = ({a^r}{b^s})\rho ({a^i}{b^j}) = {a^{r{\beta _j}{\alpha _s} - i{\beta _s}{\alpha _j}}} = {a^{ - {N_\rho }}} = {a^{{N_\lambda }}}. \qedhere
\]
\end{proof}

\begin{definition}
For each pair $ (A,B) \in {\mathbb{Z}_m} \times {\mathbb{Z}_m} $ we define a $ \mu $-map $ \mu (A,B):{D_m} \to {D_m} $ by
\[
({a^i}{b^j})\mu (A,B) = {a^{{N_\mu }}}\mbox{, where }{N_\mu } = Ai{\alpha _j} - B{\beta _j}.
\]
\end{definition}

Each $ \rho $- and $ \lambda $-map can be identified as one of these $ \mu $-maps.
\begin{lemma}
\label{l2}
For each $ r \in {\mathbb{Z}_m} $ and $ s \in {\mathbb{Z}_2} $,
\begin{enumerate}[(i)] 
	\item $ \rho ({a^r}{b^s}) = \mu ({\beta _s},\;r{\alpha _s}) ,$
	\item $ \lambda ({a^r}{b^s}) = \mu ( - {\beta _s}, - r{\alpha _s}). $
\end{enumerate}
\end{lemma}

\begin{proof}
$ ({a^i}{b^j})\mu ({\beta _s},\;r{\alpha _s}) = {a^N} $, where $ N \equiv ({\beta _s})i{\alpha _j} - (r{\alpha _s}){\beta _j} $. By Lemma \ref{l1} $ ({a^i}{b^j})\rho ({a^r}{b^s}) = {a^{{N_\rho }}} $ with
 $ {N_\rho } \equiv i{\alpha _j}{\beta _s} - r{\alpha _s}{\beta _j}(\bmod m) $. Since $ N \equiv {N_\rho } $, \emph{(i)} follows. To prove \emph{(ii)}, consider $ ({a^i}{b^j})\mu ( - {\beta _s}, - r{\alpha _s}) = {a^{N'}} $ where $ N' \equiv ( - {\beta _s})i{\alpha _j} - (r( - {\alpha _s})){\beta _j} \equiv r{\alpha _s}{\beta _j} - {\beta _s}i{\alpha _j} $. By Lemma \ref{l1} $ ({a^i}{b^j})\lambda ({a^r}{b^s}) = {a^{r{\beta _j}{\alpha _s} - i{\beta _s}{\alpha _j}}} $, thus \emph{(ii)} is established. 
\end{proof}

\begin{lemma}
\label{l3}
For each $ A,A' \in {\mathbb{Z}_m} $ and $ B,B' \in {\mathbb{Z}_2} $, $ \mu (A,B) \circ \mu (A',B') = \mu (AA',BA') $.
\end{lemma}

\begin{proof}
First note that $ ({a^i}{b^j})\mu (A,B) = {a^N} $, where $ N \equiv Ai{\alpha_j} - B{\beta _j} $. Thus
\[
({a^i}{b^j})(\mu (A,B) \circ \mu (A',B')) = ({a^N}{b^0})\mu (A',B') = {a^{N'}},
\]
where $ N' \equiv A'N{\alpha _0} - B'{\beta _0} \equiv A'N \equiv (Ai{\alpha _j} - B{\beta _j})A' \equiv AA'i{\alpha _j} - BA'{\beta _j} $.
On the other hand, $ ({a^i}{b^j})\mu (AA',BA') = {a^{N''}} $, where $ N'' \equiv AA'i{\alpha _j} - BA'{\beta _j} $. The lemma follows.
\end{proof}

Although compositions of $ \rho $-maps ( $ \lambda $-maps) may not be $ \rho $-maps ( $ \lambda $-maps), it follows from Lemma \ref{l3} that compositions of $ \mu $-maps are $ \mu $-maps. Closing the sets $ {{\rm P}_1}({D_m}) $ and $ \Lambda_1 ({D_m}) $ under multiplication to form semigroups will be facilitated by identifying the $ \rho $-maps and $ \lambda $-maps as $ \mu $-maps. Note also that the $ \mu $-maps form a subsemigroup, $ {\rm M}({D_m}) $, of $ \mathcal{M}({D_m}) $, and since $ {{\rm P}_1}({D_m}),{\Lambda _1}({D_m}) \subseteq M({D_m}) $, we have $ {\rm P}({D_m}),\Lambda ({D_m}) $ are subsemigroups of $ {\rm M}({D_m}) .$ Thus $ \left| {{\rm M}({D_m})} \right| = {m^2} $ is an upper bound on the orders of both $ {\rm P}({D_m}) $ and $ \Lambda ({D_m}) $. 

\section{Containers}
Viewed as $ \mu $-maps, the mappings in $ {\rm P}({D_m}) $ and $ \Lambda ({D_m}) $ naturally bond together into sets with one parameter running over $ {\mathbb{Z}_m} $.

\begin{definition}
If $ A,B \in {\mathbb{Z}_m} $, the \emph{$ (A,B) $-container} is defined as 
\[
	 \mathcal{C}(A,B) = \left\{ {\mu (A,xB):x \in {\mathbb{Z}_m}} \right\}.
\]
\end{definition}

For $ r \in {\mathbb{Z}_m} $ and $ s \in {\mathbb{Z}_2} $, Lemma \ref{l2} implies that $ \rho ({a^r}{b^s}) = \mu ({\beta _s},r{\alpha _s}) \in \mathcal{C}({\beta _s},r{\alpha _s}) $. Since $ r $ is a parameter running through $ {\mathbb{Z}_m} $, we can simplify this by writing $ \rho ({a^r}{b^s}) \in \mathcal{C}({\beta _s},{\alpha _s}) $. Similarly, we can write $ \lambda ({a^r}{b^s}) \in \mathcal{C}( - {\beta _s}, - {\alpha _s}) $. Thus our generating maps are in easily identified containers.

\begin{lemma}
\label{l4}
If $ A,B,B' \in {\mathbb{Z}_m} $ and $B$ is invertible in $ {\mathbb{Z}_m} $, then $ \mathcal{C}(A,BB') = \mathcal{C}(A,B'). $ 
\end{lemma}

\begin{proof}
Since $B$ is invertible in $\mathbb{Z}_m,$ $\mathbb{Z}_mB = \mathbb{Z}_m.$ That is, 
\[ \{0B,1B,\ldots, (m-1)B \} = \{ 0,1,\ldots, m-1\}.\]
We then have that
\begin{align*}
\mathcal{C}(A,BB') &= \{\mu(A,0B\cdot B'), \mu(A,1B\cdot B'),\ldots, \mu(A,(m-1)B')\} \\
&= \{ \mu(A,0B'),\mu(A,1B'),\ldots, \mu(A,(m-1)B')\} = \mathcal{C}(A,B'). \qedhere
\end{align*}
\end{proof}

\begin{lemma}
\label{l5}
If $ A,B,C \in {\mathbb{Z}_m} $, then $ \mathcal{C}(A,BC) \subseteq \mathcal{C}(A,C) $.
\end{lemma}

\begin{proof}
Let $ \mu  \in \mathcal{C}(A,BC) $, then there is an $ x \in {\mathbb{Z}_m} $ with $ \mu  = \mu (A,xBC) $. Note that
\[
\mu  = \mu (A,xBC) = \mu (A,\left( {xB} \right)C)
\]
and $ xB \in {\mathbb{Z}_m} $; therefore $ \mu  \in \mathcal{C}(A,C) $.
\end{proof}

\begin{lemma}
\label{l6}
For all $ A,A',B,B' \in {\mathbb{Z}_m} $, $ \mathcal{C}(A,B) \cap \mathcal{C}(A',B') \ne \varnothing $ if and only if $ A \equiv A'(\bmod m) $.
\end{lemma}

\begin{proof}
First suppose that $ \mathcal{C}(A,B) \cap \mathcal{C}(A',B') \ne \varnothing $. It follows that there is a mapping $ \mu  \in \mathcal{C}(A,B) $ with $ \mu  \in \mathcal{C}(A',B') $ as well. We then conclude that there exist $ x,y \in {\mathbb{Z}_m} $ such that $ \mu  = \mu (A,xB) = \mu (A',yB') $.
Thus $ (a)\mu (A,xB) = {a^{Ai{\alpha _j}}} = {a^{A(1)(1)}} = {a^A} $ and this must equal $ (a)\mu (A',yB') = {a^{A'}} $. Thus it follows that $ A \equiv A'\;(\bmod m) $. Conversely, suppose that $ A \equiv A'(\bmod m) $. Then note that $ \mu (A,0) = \mu (A,0B) \in \mathcal{C}(A,B) $, and $ \mu (A,0) = \mu (A,0B') \in \mathcal{C}(A,B') = \mathcal{C}(A',B') $ since $ A \equiv A' $. Therefore $ \mu (A,0) \in \mathcal{C}(A,B) \cap \mathcal{C}(A',B') $ and the intersection is nonempty.
\end{proof}

We can represent all $ \rho $-maps and $ \lambda $-maps as disjoint unions of containers as follows.

\begin{lemma}
\label{l7}
$ {{\rm P}_1} = \mathcal{C}(0,1)\,\dot  \cup \,\mathcal{C}( - 2,1) $ and $ {\Lambda _1} = \mathcal{C}(0,1)\,\dot  \cup \,\mathcal{C}(2,1) $.
\end{lemma}

\begin{proof}
For a fixed $ s \in {\mathbb{Z}_2} $, Lemma \ref{l2} implies that, 
\[\begin{gathered}
  \left\{ {\rho ({a^x}{b^s}):x \in {\mathbb{Z}_m}} \right\} = \left\{ {\mu ({\beta _s},x{\alpha _s}):x \in {\mathbb{Z}_m}} \right\} = \mathcal{C}({\beta _s},{\alpha _s}) , \mbox{ and } \\ 
  \left\{ {\lambda ({a^x}{b^s}):x \in {\mathbb{Z}_m}} \right\} = \left\{ {\mu ( - {\beta _s},x( - {\alpha _s})):x \in {\mathbb{Z}_m}} \right\} = \mathcal{C}( - {\beta _s}, - {\alpha _s}) .
\end{gathered}
\]
Thus
\begin{align*}
 	 {{\rm P}_1} &= \left\{ {\rho ({a^r}{b^s}):r \in {\mathbb{Z}_m},s \in {\mathbb{Z}_2}} \right\} = \left\{ {\rho ({a^r}):r \in {\mathbb{Z}_m}} \right\} \cup \left\{ {\rho ({a^r}b):r \in {\mathbb{Z}_m}} \right\} \\ 
 	 &= \mathcal{C}({\beta _0},{\alpha _0}) \cup \mathcal{C}({\beta _1},{\alpha _1}) = \mathcal{C}(0,1) \cup \mathcal{C}( - 2, - 1) .
\end{align*}

Since $- 1 $ is invertible in $ {\mathbb{Z}_m} $, Lemma \ref{l4} implies that $ \mathcal{C}( - 2, - 1) = \mathcal{C}( - 2,1) $ and the result follows. The proof for $ {\Lambda _1} $ is similar. Lemma \ref{l6} implies that both unions are disjoint.
\end{proof}

Instead of multiplying $ \rho $-maps or $ \lambda $-maps together repeatedly to find all the mappings in $ {\rm P}({D_m}) $ and $ \Lambda ({D_m}) $, we can accomplish this more efficiently by multiplying containers.

\begin{definition}
For any two containers $ \mathcal{C}(A,B) $ and $ \mathcal{C}(A',B') $, we define their product as:
\[
\mathcal{C}(A,B) \circ \mathcal{C}(A',B') = \left\{ {{\mu _1} \circ {\mu _2}:{\mu _1} \in \mathcal{C}(A,B){\text{ and }}{\mu _2} \in \mathcal{C}(A',B')} \right\}.
\]
\end{definition}
\begin{lemma}
\label{l8}
For $ A,A' \in {\mathbb{Z}_m} $ and $ B,B' \in {\mathbb{Z}_2} $, $ \mathcal{C}(A,B) \circ \mathcal{C}(A',B') = \mathcal{C}(AA',BA') $.
\end{lemma}

\begin{proof}
For $ \mu  \in \mathcal{C}(A,B) \circ \mathcal{C}(A',B') $, we have $ \mu  = {\mu _1} \circ {\mu _2} $ with $ {\mu _1} \in \mathcal{C}(A,B)$ and ${\mu _2} \in \mathcal{C}(A',B') $. Thus there exist $ x,y \in {\mathbb{Z}_m} $ such that $ {\mu _1} = \mu (A,xB) $ and $ {\mu _2} = \mu (A',yB') $. Therefore, by Lemma \ref{l3}, $ \mu  = \mu (AA',xBA') \in \mathcal{C}(AA',BA') $. On the other hand if $ \mu  \in \mathcal{C}(AA',BA') $, there is an $ x \in {\mathbb{Z}_m} $ with $ \mu  = \mu (AA',xBA') $. But, by Lemma \ref{l3} again, $ \mu (AA',xBA') = \mu (A,xB) \circ \mu (A',B') $, and it is clear that $ \mu (A,xB) \in \mathcal{C}(A,B) $ and $ \mu (A',B') = \mu (A',(1)B') \in \mathcal{C}(A',B') $. Therefore, $ \mu  \in \mathcal{C}(A,B) \circ \mathcal{C}(A',B') $ and the proof is complete.
\end{proof}

Although we will not use the fact here, it is interesting to note that the set of all containers endowed with this product forms a quotient semigroup of $ M({D_m}) $.

By Lemma \ref{l7}, $ {{\rm P}_1} $ is the union of two containers. Forming all possible products with these two containers will produce containers holding all possible products of $ \rho $-maps and, thus, express $ {\rm P}({D_m}) $ as a union of containers. Similarly the two containers holding the maps in $ {\Lambda _1} $ can be multiplied repeatedly to generate $ \Lambda ({D_m}) $. Before we generate these expressions, we take a few  more preliminary steps.

\begin{lemma}
\label{l9}
For $ A,B \in {\mathbb{Z}_m} $, 
\begin{enumerate}[(i)]
	\item $ \mathcal{C}(0,1) \circ \mathcal{C}(A,B) \subseteq \mathcal{C}(0,1),$
	\item $ \mathcal{C}(A,B) \circ \mathcal{C}(0,1) \subseteq \mathcal{C}(0,1) $.
\end{enumerate}
\end{lemma}

\begin{proof}
To establish part \emph{(i)}, note that by Lemma \ref{l8}, 
\[
	 \mathcal{C}(0,1) \circ \mathcal{C}(A,B) = \mathcal{C}(0 \cdot A,1 \cdot A) = \mathcal{C}(0,A) \subseteq \mathcal{C}(0,1) .
\]
The last containment follows by Lemma \ref{l5}. Part \emph{(ii)} also follows by Lemmas \ref{l8} and \ref{l5}:
\[
  \mathcal{C}(A,B) \circ \mathcal{C}(0,1) = \mathcal{C}(A \cdot 0,B \cdot 0) = \mathcal{C}(0,0) \subseteq \mathcal{C}(0,1).\qedhere
\]
\end{proof}

The following is a simple fact about congruences which we will use in several of the arguments below.

\begin{lemma}
\label{l10}
If $ x $ and y are positive integers, then for every $ u,v \in \mathbb{Z}, $  $ xu \equiv xv\;(\bmod xy) $ if and only if $ u \equiv v\;(\bmod y) $.
\end{lemma}

\begin{proof}
Suppose first that $ xu \equiv xv\;(\bmod xy) $. Then $ x(u - v) \equiv 0\;(\bmod xy) $. Thus there  exists a $ t \in \mathbb{Z} $ such that $(*)$ $ x(u - v) = txy $. If $ u - v = 0 $, then $ u = v $ and it follows that $ u \equiv v\;(\bmod y) $. If $ u - v \ne 0 $ then we obtain $ u - v = ty $ by cancellation of x from both sides of $(*)$. Thus $ u \equiv v\;(\bmod y) $. Conversely, suppose that $ u \equiv v\;(\bmod y) $. It follows that there exists $ t \in \mathbb{Z} $ such that $ u - v = ty $. Thus $ x(u - v) = txy $. Therefore $ xu \equiv xv\;(\bmod xy) $.
\end{proof}

It is well known that if $ S $ is a semigroup and $ x \in S $, the positive powers of  x, $ \left\{ {{x^i}:i \in {\mathbb{Z}^ + }} \right\} $, form a subsemigroup of $S$ called the $ monogenic$ $ semigroup $ generated by $x$, denoted $ \left\langle x \right\rangle $. If $S$ is finite, then for each $ x \in S $ the powers of $x$ begin to repeat at some point. Let $c$ be the smallest positive integer such that there is a $ k \in {\mathbb{Z}^ + } $ with $ {x^c} = {x^{c + k}}\; $ and $ k $ the least such positive integer. Here $c$ and $k$ are called the $ index $ and the $ period $ of $x$, and $ \left\langle x \right\rangle  = \left\{ {{x^1}, \ldots ,{x^{c + k - 1}}} \right\} $ is the set of distinct powers of $x$. When $ S = {\mathbb{Z}_m} $ we denote the index and period of $ x \in {\mathbb{Z}_m} $, by $ in{d_m}(x) $ and $ pe{r_m}(x) $ respectively. If $ x \not \equiv 0\ $ and $ in{d_m}(x) = 1 $ then
 $ {x^{pe{r_m}(x)}} \equiv 1\;(\bmod m) $, $ \left\langle x \right\rangle $ is the cyclic group of order $ pe{r_m}(x) $, and this happens exactly when $x$ is invertible in $ \mathbb{Z}_m $. In this case we denote the order of $x$ as an element of $ \left( {{\mathbb{Z}_m}, \cdot } \right) $ by $ or{d_m}(x) $.

The next result characterizes the indices for $- 2 $ and $ 2 $ in $ {\mathbb{Z}_m} $ in  terms of the number theoretic form of $m$. In some cases information about the period can also be given.

\begin{lemma}
\label{l11}
If $ m = {2^\ell }n \geq 3 $ with $n$ odd, $ \ell  \geq 0 $, and $ n \geq 1 $, then for $ x \in \left\{ { - 2,2} \right\} $ 
\begin{enumerate}[(i)]
	\item if $m$ is odd, then $ in{d_m}(x) = 1 $ and $ pe{r_m}(x) = or{d_m}(x) $,
	\item if $m$ is even and $ n > 1 $, then $ in{d_m}(x) = \ell $,
	\item if $m$ is even and $ n = 1 $, then $ in{d_m}(x) = \ell $ and $ pe{r_m}(x) = 1 $.
\end{enumerate}
\end{lemma}

\begin{proof}
We will give proofs of each part only for $ x =  - 2 $, since the proofs for $ x = 2 $ are quite similar.
\begin{proof}[Proof of (i)]
With $m$ odd, it follows that $- 2 $ is invertible in $ {\mathbb{Z}_m} $; thus it has an order, say $ or{d_m}( - 2) = k $. Thus we have $ 1 \equiv {( - 2)^k} $ and hence, $ {( - 2)^1} = {( - 2)^{1 + k}} $ and part \emph{(i)} of the lemma follows easily. 
\noqed
\end{proof}
\begin{proof}[Proof of (ii)]
With $m$ even $ \left( {\ell  > 0} \right) $ and $ n > 1 $, we know that $- 2 $ has an index and a period, say $ in{d_m}( - 2) = c $ and $ pe{r_m}( - 2) = k $. We wish to show that $ c = \ell $. First suppose that $ c < \ell $. We can rewrite $ {( - 2)^c} \equiv {( - 2)^{c + k}}\;(\bmod m) $ as 
 $ {( - 1)^c}{2^c} \equiv {( - 1)^c}{2^c}{( - 2)^k}\;(\bmod {2^c}{2^{\ell  - c}}n) $.

Applying  Lemma \ref{l10} we have, $ {( - 1)^c} \equiv {( - 1)^c}{( - 2)^k}\;(\bmod {2^{\ell  - c}}n) $. Multiplying on both sides by $ {( - 1)^c} $ we obtain $ 1 \equiv {( - 2)^k}\;(\bmod {2^{\ell  - c}}n) $. But this implies that $- 2 $ is invertible modulo $ {2^{\ell  - c}}n $ with $ \ell  - c > 0 $. Since $ \gcd ( - 2,{2^{\ell  - c}}n) = 2 \ne 1 $, $- 2 $ is not invertible, and we have a contradiction. Now suppose that $ c > \ell $. Since $ n > 1 $ is odd, $ \gcd (n, - 2) = 1 $ and, therefore, $- 2 $ is invertible modulo $n$. Letting $ r = or{d_n}( - 2) $, we have $ 1 \equiv {( - 2)^r}(\bmod n) $. By Lemma \ref{l10}, $ {2^\ell } = {2^\ell }{( - 2)^k}\;(\bmod {2^\ell }n) $. Multiplying on both sides by $ {( - 1)^\ell } $ gives $ {( - 2)^\ell }\; \equiv {( - 2)^{\ell  + k}}(\bmod m) $. But since $ c > \ell $, this contradicts the minimality of the originally selected index, $c$. It follows that $ \ell  = c = in{d_m}( - 2) $.
\noqed
\end{proof}
\begin{proof}[Proof of (iii)]
In this case we have $ m = {2^\ell } $. Note that 
\[
 {( - 2)^\ell } \equiv {( - 1)^\ell }{2^\ell } \equiv 0\;(\bmod {2^\ell }) 
\]
and that for each $ t \in {\mathbb{Z}^ + } $, $ {( - 2)^{\ell  + t}} \equiv 0\;(\bmod {2^\ell }) $; thus $ {( - 2)^\ell } = {( - 2)^{\ell  + 1}} $. We will argue that $ in{d_{{2^\ell }}}( - 2) = \ell $ and, consequently, $ pe{r_{{2^\ell }}}( - 2) = 1 $. We will first prove that the terms in the sequence $ \left( s \right) = \left( {{{( - 2)}^1},{{( - 2)}^2}, \ldots ,{{( - 2)}^\ell }} \right) $ 
are distinct modulo $ {2^\ell } $. Once this is established, we will know that $ {( - 2)^\ell } $ is the only power of $- 2 $ in $ \left( s \right) $ which is congruent to 0. Thus if a higher power of $- 2 $ is congruent to a term of $ \left( s \right) $, it is congruent to $ {( - 2)^\ell } $ only. Once the terms of $ \left( s \right) $ are shown to be distinct, it is then impossible for $ in{d_m}( - 2) < \ell $, and the result follows.

To show that the terms in $ \left( s \right) $ are distinct, suppose that $ 0 < u < v \leq \ell $, but that $ {( - 2)^u} \equiv {( - 2)^v}\;(\bmod {2^\ell }) $. We can rewrite this as 
\[
 {( - 1)^u}{2^u} \equiv {( - 1)^u}{2^u}{( - 2)^{v - u}}\;(\bmod {2^u}{2^{\ell  - u}}) 
\]
and conclude, by Lemma \ref{l10}, that $ {( - 1)^u} \equiv {( - 1)^u}{( - 2)^{v - u}}\;(\bmod {2^{\ell  - u}}) $. Multiplying on both sides by $ {( - 1)^u} $, we get $ 1 \equiv {( - 2)^{v - u}}\;(\bmod {2^{\ell  - u}}) $ 
and, hence, $- 2 $ is invertible modulo $ {2^{\ell  - u}} $. Since and $ u < v \leq \ell $ we see that $ \gcd ( - 2,{2^{\ell  - u}}) = 2 \ne 1 $ which contradicts the invertibility of $- 2 $. Thus the powers of $- 2 $ in $ \left( s \right) $ are distinct.
\end{proof} \noqed
\end{proof}

We are now able to express the right and left commutation semigroups as disjoint unions of containers.

\begin{theorem}
\label{t12}
For $ m = {2^\ell }n \geq 3 $ with n odd, $ \ell  \geq 0 $, and $ n \geq 1 $,
\begin{enumerate}[(i)]
	\item $ 
	 {\rm P}({D_m}) = \mathcal{C}(0,1) \cup \left( {\bigcup\limits_{i = 1}^t {\mathcal{C}({{( - 2)}^i},{{( - 2)}^{i - 1}})} } \right){\text{, where }} 
	  t = \left\{ {
	  \begin{array}{*{20}{c}}
  		{or{d_m}( - 2){\text{ for }}\ell  = 0,\,n > 1} \\ 
 		{\ell  + pe{r_m}( - 2) - 1{\text{ for }}\ell  > 0,\,n > 1} \\ 
 		{\ell  - 1{\text{ for }}\ell  > 0,\,n = 1} 
	\end{array}} \right.$
	\item $
	\Lambda ({D_m}) = \mathcal{C}(0,1) \cup \left( {\bigcup\limits_{i = 1}^{t'} {\mathcal{C}({2^i},{2^{i - 1}})} } \right){\text{, where }}  
	t' = \left\{ {
	\begin{array}{*{20}{c}}
 		{or{d_m}(2){\text{ for }}\ell  = 0,\, n > 1} \\ 
  		{\ell + pe{r_m}(2) - 1{\text{ for }}\ell  > 0,\, n > 1} \\ 
  		{\ell  - 1{\text{ for }}\ell  > 0,\, n = 1} 
	\end{array}} \right. ,$
\end{enumerate}
and these unions are disjoint.
\end{theorem}
\begin{proof}[Proof of (i)]
As mentioned earlier, $ {\rm P}({D_m}) $ is generated by repeated multiplication of the containers
 $ \mathcal{C}(0,1) $ and $ \mathcal{C}( - 2,1) $. Lemma \ref{l9} shows that the $ \mu $-maps in any container produced by a product having $ \mathcal{C}(0,1) $ as a factor are already in $ \mathcal{C}(0,1) $; thus, after we include $ \mathcal{C}(0,1) $ in the union, we need only use $ \mathcal{C}( - 2,1) $ in forming these products. To this end, we list the sequence of ``powers'' of $ \mathcal{C}( - 2,1) $ taken according to Lemma \ref{l8}:
\begin{equation}
\label{e31}
\mathcal{C}( - 2,1),\mathcal{C}({( - 2)^2}, - 2),\mathcal{C}({( - 2)^3},{( - 2)^2}), \ldots ,\mathcal{C}({( - 2)^i},{( - 2)^{i - 1}}), \ldots .
\end{equation}
We now divide the argument into three cases: (a) $m$ odd, (b) $ m = {2^\ell }n $ with $n$ odd and $ n > 1 $, (c) $ m = {2^\ell } $.

\begin{proof}[Case (a)]
Applying Lemma \ref{l11}\emph{(i)} to the first coordinates of the containers in (\ref{e31}), we see that there is a first repeat of the powers of $- 2 $ when $ {( - 2)^1} \equiv {( - 2)^{1 + d}} $ with $ d = or{d_m}( - 2) $. Keeping an eye on the second coordinates we pass from the last container in the union given in Theorem \ref{t12}\emph{(i)}, $ \mathcal{C}({( - 2)^d},{( - 2)^{d - 1}}) $, to the next container, $ \mathcal{C}({( - 2)^{d + 1}},{( - 2)^d}) $, by forming the product
\[
\mathcal{C}({( - 2)^d},{( - 2)^{d - 1}}) \circ \mathcal{C}( - 2,1) = \mathcal{C}({( - 2)^{d + 1}},{( - 2)^d}) = \mathcal{C}( - 2,1) .
\]
Clearly we did not need to include $ \mathcal{C}({( - 2)^{d + 1}},{( - 2)^d}) $ or any further containers in our union, since they will already appear earlier. This establishes \emph{(i)} for Case {(a)}.
\noqed
\end{proof}
\begin{proof}[Case (b)]
In the same way, we apply Lemma \ref{l11}\emph{(ii)} to find that $ in{d_m}( - 2) \equiv \ell $. Therefore we have $ {( - 2)^\ell } \equiv {( - 2)^{\ell  + d}} $ where $ d = pe{r_m}( - 2) $. Since $ {( - 2)^{\ell  + d}} \equiv {( - 2)^\ell } $, when we form the product of $ \mathcal{C}( - 2,1) $ with $ \mathcal{C}({( - 2)^{\ell  + d - 1}},{( - 2)^{\ell  + d - 2}}) $, the last container mentioned in Theorem \ref{t12}\emph{(i)}, we have
\begin{align*}
\mathcal{C}({( - 2)^{\ell  + d - 1}},{( - 2)^{\ell  + d - 2}}) \circ \mathcal{C}( - 2,1) &= \mathcal{C}({( - 2)^{\ell  + d}},{( - 2)^{\ell  + d - 1}})\\
&= \mathcal{C}({( - 2)^\ell },{( - 2)^{\ell  - 1}})
\end{align*}
As in the previous case, this container, and subsequent containers from (\ref{e31}), are already present in (\ref{e31}) and, thus, we may stop forming containers at this point.
\noqed
\end{proof}
\begin{proof}[Case (c)]
By Lemma 11\emph{(iii)}, we have $ {( - 2)^\ell } = {( - 2)^{\ell  + 1}} $. In this case, the last container listed in the union in (3.1) is $ \mathcal{C}({( - 2)^{\ell  - 1}},{( - 2)^{\ell  - 2}}) $. Forming the next product we obtain
\begin{align*}
\mathcal{C}({( - 2)^{\ell  - 1}},{( - 2)^{\ell  - 2}}) \circ \mathcal{C}( - 2,1) &= \mathcal{C}({( - 2)^\ell },{( - 2)^{\ell  - 1}}) \\
&= \mathcal{C}(0,{( - 2)^{\ell  - 1}}) \subseteq \mathcal{C}(0,1) .
\end{align*}
Thus this and subsequent containers are contained within $ \mathcal{C}(0,1) $ and we are able to stop the procedure with $ \mathcal{C}({( - 2)^{\ell  - 1}},{( - 2)^{\ell  - 2}}) $ as the last container in the union.
\noqed
\end{proof}
Note that in any case, Lemma \ref{l6} implies that all the containers in the union, Theorem \ref{t12}\emph{(i)}, are disjoint.
\noqed
\end{proof}
\begin{proof}[Proof of (ii)]
The proof for $ \Lambda ({D_m}) $ is essentially the same. We generate the ``powers'' of $ \mathcal{C}(2,1) $ and apply the same reasoning, quoting Lemma \ref{l11} at the appropriate points.
\end{proof}

Since both unions in Theorem \ref{t12} are disjoint, we can calculate the orders of these semigroups if we know the cardinality of each container and the periods of $- 2 $ and $ 2 $ in $ {\mathbb{Z}_m} $. There is no known formula for these periods, so they must be calculated individually for each value of $m$; however, in the next section, we will calculate the cardinality of each container. 

\section{The Cardinality of $ \mathcal{C}(A,B) $ }
Since $ \mathcal{C}(A,B) = \left\{ {\mu (A,xB):x \in {Z_m}} \right\} $, it contains at most $m$ distinct mappings. Our goal here is to determine which, if any, of these mappings are equal. This will then allow us to find the cardinality of each container. Our first task is to find the upper central series of $ {D_m} $. 

\begin{definition}
If $G$ is a group, the \emph{left normed commutator of weight $ w \geq 2 $} with entries $ {g_1},{g_2}, \ldots ,{g_w} \in G $ is the iterated commutator $ [ \ldots [[{g_1},{g_2}],{g_3}], \ldots ,{g_w}] $ written, more simply, as $ [{g_1},{g_2},{g_3}, \ldots ,{g_w}] $. In the special case of a repeated entry, we write $ [x,(n)y] = [x,y,y, \ldots ,y]\;(n{\text{ times)}} $.
\end{definition}
\begin{definition}
The \emph{upper central series of a group $G$} is the series of subgroups of $G$, 
\[
{Z_0}(G) \leq {Z_1}(G) \leq  \cdots  \leq {Z_n}(G) \leq  \cdots 
\]
with $ {Z_0}(G) = \left\{ 1 \right\} $ and $ {Z_n}(G) = \left\{ {g \in G:[g,{g_1},{g_2}, \ldots ,{g_n}] = 1,{\mbox{ for all }}{g_1},{g_2}, \ldots ,{g_n} \in G} \right\} $. We call $ {Z_n}(G) $ the $n$-th-centre of $G$ and, where no ambiguity arises, denote it $ {Z_n} $.
\end{definition}
Each term of the upper central series is normal in $G$. For $G$ finite, there is a least possible $ c \geq 0 $ for which
 $ {Z_c} = {Z_{c + 1}} = {Z_{c + 2}} =  \cdots $. If the upper central series reaches $G$ (i.e. $ {Z_c} = G $ ), then we say $G$ is nilpotent of class $ c $; otherwise $G$ is non-nilpotent and $ {Z_c}(G) < G $. Among dihedral groups only the 2-groups, $ {D_{{2^\ell }}} $  $ (\ell  \in {\mathbb{Z}^ + }) $, are nilpotent and, in this case, $ \left\{ 1 \right\} = {Z_0}({D_{{2^\ell }}}) < {Z_1}({D_{{2^\ell }}}) <  \cdots  < {Z_\ell }({D_{{2^\ell }}}) = {D_{{2^l}}} $.

The following characterization of the terms of the upper central series for $ {D_m} $ is well-known and proved routinely.

\begin{theorem}
\label{t14} \mbox{}
\begin{enumerate}[(a)]
	\item If $ u \geq 0 $ and $ m $ is odd, then $ {Z_u}({D_m}) = \left\{ 1 \right\}, $
	\item If $ u \geq 0 $ and $ m $ is even with $ m = {2^\ell }n $  $ (n > 0{\text{ and }}n{\text{ odd}}) $, then
	\begin{enumerate}[(i)]
		\item if $ n > 1 $, then \[{Z_u}({D_m}) = \left\{ {\begin{array}{*{20}{c}}
  {\left\{ {{a^N}:N = ({2^{\ell  - u}}n)x{\text{ and }}0 \leq x < {2^u}} \right\}{\text{   for }}u < \ell } \\ 
  {\left\{ {{a^{nx}}:0 \leq x < {2^\ell }} \right\}{\text{                               for }}u \geq \ell } 
\end{array}} \right.,\]
		\item if $ n = 1 $, then \[ {Z_u}({D_m}) = \left\{ {\begin{array}{*{20}{c}}
  {\left\{ {{a^N}:N = {2^{\ell  - u}}x{\text{ and }}0 \leq x < {2^u}} \right\}{\text{       for }}u < \ell } \\ 
  {{D_m}{\text{                                                   for }}u \geq \ell } 
\end{array}} \right. . \] 
	\end{enumerate}\qed
\end{enumerate}
\end{theorem}

We will adopt exponential notation for repeated composition of mappings. For any mapping $ \mu $ and any $ t \in {\mathbb{Z}^ + } $, we will write $ {\mu ^t} = \mu  \circ \mu  \circ  \cdots  \circ \mu \;(t{\text{ times)}} $, and $ {\mu ^0} $ for the identity mapping on $ {D_m} $. The following lemma takes us back from the $ \mu  {\text{-maps}} $ in our containers to the $ \rho {\text{-maps}} $ they represent. 

\begin{lemma}
\label{l15}
For $ u > 0 $ and $ x \in {\mathbb{Z}_m} $,
\begin{enumerate}[(i)] 
	\item $ \mu ({( - 2)^u},x{( - 2)^{u - 1}}) = \rho ({a^{ - x}}b) \circ \rho {(b)^{u - 1}} ,$ 	
	\item $ \mu ({2^u},x{2^{u - 1}}) = \lambda ({a^{ - x}}b) \circ \lambda {(b)^{u - 1}}. $
\end{enumerate}
\end{lemma}

\begin{proof}
To prove \emph{(i)} we will apply both maps to an arbitrary $ {a^i}{b^j} \in {D_m} $. Note from Lemma \ref{l2} that,
 $ \rho ({a^{ - x}}b) = \mu ({\beta _1}, - x{\alpha _1}) = \mu ( - 2,x) $. Therefore, $ ({a^i}{b^j})\rho ({a^{ - x}}b) = ({a^i}{b^j})\mu ( - 2,x) = {a^N} $ with $ N =  - 2i{\alpha _j} - x{\beta _j} $. Note also that $ \rho (b) = \mu ( - 2,0) $. Thus, by use of Lemma \ref{l3} repeatedly, we have $ ({a^i}{b^j})\rho ({a^{ - x}}b) \circ \rho {(b)^{u - 1}} = {a^{N'}} $ where $ N' = N{( - 2)^{u - 1}} $. Now $ ({a^i}{b^j})\mu ({( - 2)^u},x{( - 2)^{u - 1}}) = {a^{N''}} $ where $ N'' = {( - 2)^u}i{\alpha _j} - x{( - 2)^{u - 1}}{\beta _j} = ( - 2i{\alpha _j} - x{\beta _j}){( - 2)^{u - 1}} = N{( - 2)^{u - 1}} = N' $ and the result follows. The proof of \emph{(ii)} follows similarly.
\end{proof}

A group $G$ is metabelian (or solvable of length 2) if its commutator subgroup, $ G' $, is abelian. The dihedral groups are metabelian since $ {D'_m} \leq \left\langle a \right\rangle $,which is cyclic, and hence abelian. The following commutator identities, which hold in any metabelian group, will be used in our arguments. 
\begin{samepage}
\begin{atheorem}{MCI} \mbox{}
\begin{enumerate}[(i)]
	\item $ [xy,{z_1},{z_2}, \ldots ,{z_n}] = [{[x,{z_1}]^y},{z_2}, \ldots ,{z_n}][y,{z_1},{z_2}, \ldots ,{z_n}], $ 
	\item $ [[x,y][u,v],{z_1},{z_2}, \ldots ,{z_n}] = [x,y,{z_1},{z_2}, \ldots ,{z_n}][u,v,{z_1},{z_2}, \ldots ,{z_n}] ,$ 
	\item $ [{[x,y]^{ - 1}},{z_1},{z_2}, \ldots ,{z_n}] = {[x,y,{z_1},{z_2}, \ldots ,{z_n}]^{ - 1}} .$ \qed
\end{enumerate}
\end{atheorem}
\end{samepage}
The following theorem is true for any metabelian group and, therefore, holds for $ {D_m} $.

\begin{theorem}
\label{t16}
Let $G$ be a metabelian group with $ u > 0 $ and $ {g_1},{g_2} \in G $, then 
\begin{equation}
\label{e46}
[{g_1},{x_1},{x_2}, \ldots ,{x_u}] = [{g_2},{x_1},{x_2}, \ldots ,{x_u}]
\end{equation}
for every $ {x_1},{x_2}, \ldots ,{x_u} \in G $ if and only if $ g_1^{ - 1}{g_2} \in {Z_u}(G) $.
\end{theorem}

\begin{proof}
$ ( \Rightarrow ) $ To show that $ g_1^{ - 1}{g_2} \in {Z_u}(G) $ it suffices to verify that
\[
[g_1^{ - 1}{g_2},{x_1},{x_2}, \ldots ,{x_u}] = 1 . 
\]
Therefore,
\begin{align*}
	 [g_1^{ - 1}{g_2},{x_1},{x_2}, \ldots ,{x_u}] &= [[g_1^{ - 1},{x_1}][g_1^{ - 1},{x_1},{g_2}][{g_2},{x_1}],{x_2}, \ldots ,{x_u}]\mbox{ by CI (iii)} \\
	&= [g_1^{ - 1},{x_1},{x_2}, \ldots ,{x_u}][[g_1^{ - 1},{x_1}],{g_2},{x_2}, \ldots ,{x_u}][{g_2},{x_1},{x_2}, \ldots ,{x_u} \mbox{ by MCI (ii)}\\
	&= [g_1^{ - 1},{x_1},{x_2}, \ldots ,{x_u}]{[{g_2},[g_1^{ - 1},{x_1}],{x_2}, \ldots ,{x_u}]^{ - 1}}[{g_2},{x_1},{x_2}, \ldots ,{x_u}] \\
	&\quad\mbox{ by CI (ii) and MCI (iii)}. 
\end{align*}
Similarly,
\begin{align*}
	1 &= [g_1^{ - 1}{g_1},{x_1},{x_2}, \ldots ,{x_u}] = [[g_1^{ - 1},{x_1}][g_1^{ - 1},{x_1},{g_1}][{g_1},{x_1}],{x_2}, \ldots ,{x_u}] \\
	&= [g_1^{ - 1},{x_1},{x_2}, \ldots ,{x_u}][g_1^{ - 1},{x_1},{g_1},{x_2}, \ldots ,{x_u}][{g_1},{x_1},{x_2}, \ldots ,{x_u}]\\ 
	&= [g_1^{ - 1},{x_1},{x_2}, \ldots ,{x_u}]{[{g_1},[g_1^{ - 1},{x_1}],{x_2}, \ldots ,{x_u}]^{ - 1}}[{g_1},{x_1},{x_2}, \ldots ,{x_u}] .
\end{align*}
Comparing the last lines of these two calculations, we see that the first factors are identical, the second factors are equal using the hypothesis (\ref{e46}) with $ {x_1} $ replaced by $ [g_1^{ - 1},{x_1}] $, and the third factors are equal by (\ref{e46}). Thus we obtain $ [g_1^{ - 1}{g_2},{x_1},{x_2}, \ldots ,{x_u}] = 1 $, and the implication follows.

$ \left(  \Leftarrow  \right) $ Since $ g_1^{ - 1}{g_2} \in {Z_u}(G) $, we know that for every $ {x_1},{x_2}, \ldots ,{x_u} \in G $, $ [g_1^{ - 1}{g_2},{x_1},{x_2}, \ldots ,{x_u}] = 1 $. Therefore,
\begin{align*}
	1 &= [g_1^{ - 1}{g_2},{x_1},{x_2}, \ldots ,{x_u}] \\ 		
	&= [{[g_1^{ - 1},{x_1}]^{{g_2}}},{x_2}, \ldots ,{x_u}][[{g_2},{x_1}],{x_2}, \ldots ,{x_u}]\mbox{ by MCI (i)}\\
	&= [{[{g_1},{x_1}]^{ - g_1^{ - 1}{g_2}}},{x_2}, \ldots ,{x_u}][{g_2},{x_1},{x_2}, \ldots ,{x_u}]\mbox{ by CI (v)}\\
	&= {[{[{g_1},{x_1}]^{g_1^{ - 1}{g_2}}},{x_2}, \ldots ,{x_u}]^{ - 1}}[{g_2},{x_1},{x_2}, \ldots ,{x_u}]\mbox{ by MCI (iii)}\\
	&= {[[{g_1},{x_1}][{g_1},{x_1},g_1^{ - 1}{g_2}],{x_2}, \ldots ,{x_u}]^{ - 1}}[{g_2},{x_1},{x_2}, \ldots ,{x_u}]\mbox{ by CI (i)}\\
	&= {([[{g_1},{x_1}],{x_2}, \ldots ,{x_u}][[[{g_1},{x_1}],g_1^{ - 1}{g_2}],{x_2}, \ldots ,{x_u}])^{ - 1}}[{g_2},{x_1},{x_2}, \ldots ,{x_u}]\mbox{ by MCI (ii)}\\
	&= ([[{g_1},{x_1}],{x_2}, \ldots ,{x_u}]{[[g_1^{ - 1}{g_2},[{g_1},{x_1}],{x_2}, \ldots ,{x_u}])^{ - 1}}[{g_2},{x_1},{x_2}, \ldots ,{x_u}]\mbox{ by CI(ii) and MCI(iii)}\\
\end{align*}
Since $ g_1^{ - 1}{g_2} \in {Z_u}(G) $, we know that the middle term in the last line is trivial. Thus we have
			 \[1 = {[{g_1},{x_1},{x_2}, \ldots ,{x_u}]^{ - 1}}[{g_2},{x_1},{x_2}, \ldots ,{x_u}].\]
Therefore, $ [{g_1},{x_1},{x_2}, \ldots ,{x_u}] = [{g_2},{x_1},{x_2}, \ldots ,{x_u}] $, as required.
\end{proof}

\begin{lemma}
\label{l13}
For $ u > 0 $ and $ 1 \leq j \leq u $ let $ i,{r_j} \in {\mathbb{Z}_m},\;{s_j} \in {\mathbb{Z}_2} $, then
 \[[{a^i},{a^{{r_1}}}{b^{{s_1}}},{a^{{r_2}}}{b^{{s_2}}}, \ldots ,{a^{{r_u}}}{b^{{s_u}}}] = \left\{ {\begin{array}{*{20}{c}}
  {{\text{  1                  if some }}{s_k} = 0} \\ 
  {[{a^i},(u - 1)b]{\text{   if all }}{s_k} = 1{\text{ }}} 
\end{array}} \right\}.
\]
Note that if all $ {s_k} = 1 $, then 
\[[{a^i},(u - 1)b] = {a^N}{\text{ with }}N \equiv {( - 2)^u}i.\]
\end{lemma}

\begin{proof}
Note first that for $ i,r \in {\mathbb{Z}_m},\;s \in {\mathbb{Z}_2} $, CI\emph{(iv)} implies that $ [{a^i},{a^r}{b^s}] = [{a^i},{b^s}]{[{a^i},{a^r}]^{{b^s}}} = [{a^i},{b^s}] $.
It follows that \[[{a^i},{a^{{r_1}}}{b^{{s_1}}}] = [{a^i},{b^{{s_1}}}] = {a^{ - i}}{({a^i})^{{b^{{s_1}}}}} = {a^{ - i}}{({a^i})^{{{( - 1)}^{{s_1}}}}} = {a^{{\beta _{{s_1}}}i}}.\] Similarly,
\[
 [{a^i},{a^{{r_1}}}{b^{{s_1}}},{a^{{r_2}}}{b^{{s_2}}}] = [{a^{{\beta _{{s_1}}}i}},{a^{{r_2}}}{b^{{s_2}}}] = [{a^{{\beta _{{s_1}}}i}},{b^{{s_2}}}] = {a^{{\beta _{{s_1}}}{\beta _{{s_2}}}i}} ,
\]
and, inductively, $ [{a^i},{a^r}{b^s},{a^{{r_1}}}{b^{{s_1}}}, \ldots ,{a^{{r_{u - 1}}}}{b^{{s_{u - 1}}}}] = {a^N} $ with $ N = {\beta _{{s_1}}} \cdots {\beta _{{s_u}}}i $. If some $ {s_j} = 0 $, then $ {\beta _{{s_j}}} = 0 $, and it follows that $ N \equiv 0\;(\bmod m) $. In this case we have $ {a^N} = 1 $ as stated in the lemma. Otherwise, $ {s_j} = 1 $ for each $ j $. Thus, $ {\beta _j} =  - 2 $ for each $j$, and therefore $ {a^N} = {( - 2)^u}i $, as required.
\end{proof}

\begin{lemma}
\label{c17}
Let $ m \geq 3 $, and $ u > 0 $ and $ {g_1},{g_2} \in {D_m} $, then 
\begin{equation}
\label{e49}
[{g_1},{x_1},{x_2}, \ldots ,{x_u}] = [{g_2},{x_1},{x_2}, \ldots ,{x_u}] \mbox{ for every } {x_1},{x_2}, \ldots ,{x_u} \in {D_m}
\end{equation}
if and only if
\begin{equation}
\label{e410}
\rho ({g_1}) \circ \rho {(b)^{u - 1}} = \rho ({g_2}) \circ \rho {(b)^{u - 1}}.
\end{equation}
\end{lemma}

\begin{proof}
$ \left(  \Rightarrow  \right) $ Let $ x \in {D_m} $, then
\begin{align*}
	&(x)\rho ({g_1}) \circ \rho {(b)^{u - 1}} \\ 
	=& [x,{g_1},(u - 1)b] \\ 
	=& {[{g_1},x,(u - 1)b]^{ - 1}} \mbox{ by CI (ii) and MCI (iii)}\\
	=& {[{g_2},x,(u - 1)b]^{ - 1}} \mbox{ by (\ref{e49})}\\
	=& [x,{g_2},(u - 1)b] \mbox{ by CI (ii) and MCI (iii)}\\
	=& (x)\rho ({g_2}) \circ \rho {(b)^{u - 1}} .
\end{align*}
$ \left(  \Leftarrow  \right) $ Letting $ [{g_1},{x_1}] = {a^i} $ for some $ i \in {\mathbb{Z}_m} $ and writing each $ {x_k} $ as $ {a^{{r_k}}}{b^{{s_k}}} $, we obtain  
\begin{align*}
[{g_1},{x_1},{x_2}, \ldots ,{x_u}] &= [{a^i},{a^{{r_2}}}{b^{{s_2}}}, \ldots ,{a^{{r_u}}}{b^{{s_u}}}] \\ 	
	&= \left\{ {\begin{array}{*{20}{c}}
 	 {1{\text{          if some }}{s_k} = 0} \\ 
 	 {[{a^i},(u - 1)b]{\text{ if }}{s_k} = 1{\text{ }}\left( {2 \leq k \leq u} \right)} 
\end{array}} \right\} ,
\end{align*}
by Lemma \ref{l13}. Similarly, if $ [{g_2},{x_1}] = {a^j} $, then 
\[
[{g_2},{x_1},{x_2}, \ldots ,{x_u}] = \left\{ {\begin{array}{*{20}{c}}
  {1{\text{          if some }}{s_k} = 0} \\ 
  {[{a^j},(u - 1)b]{\text{ if }}{s_k} = 1{\text{ }}\left( {2 \leq k \leq u} \right)} 
\end{array}} \right\} .
\]
If $ {s_k}{\kern 1pt}  = 0 $ for some $ k $, then both expressions equal 1, and (\ref{e49}) holds.  Thus we may suppose that $ {s_k}{\kern 1pt}  = 1 $ for $ 2 \leq k \leq u $. Here the equation in (\ref{e49}) becomes $ [{a^i},(u - 1)b] = [{a^j},(u - 1)b] $. Replacing $ {a^i}{\text{ by }}[{x_1},{g_1}] $ and $ {a^j}{\text{ by }}[{x_1},{g_2}] $, we obtain, $ {[[{x_1},{g_1}],(u - 1)b]^{ - 1}} = {[[{x_1},{g_2}],(u - 1)b]^{ - 1}} $. This is equivalent to $ ({x_1})(\rho ({g_1}) \circ \rho {(b)^{u - 1}}) = ({x_1})(\rho ({g_2}) \circ \rho {(b)^{u - 1}}) $ which then follows from our assumption (\ref{e410}).
\end{proof}

\begin{corollary}
\label{c18}
If $ u > 0 $, $ {Z_u}({D_m}) \leq \left\langle a \right\rangle $, and $ x,y \in {Z_m} $, then 
 $ \rho ({a^x}b) \circ {\left( {\rho (b)} \right)^{u - 1}} = \rho ({a^y}b) \circ {\left( {\rho (b)} \right)^{u - 1}} $ if and only if $ {a^x}{Z_u} = {a^y}{Z_u} $ in the quotient group $ {{\left\langle a \right\rangle } \mathord{\left/
 {\vphantom {{\left\langle a \right\rangle } {{Z_u}}}} \right.
 \kern-\nulldelimiterspace} {{Z_u}}} $.
\end{corollary}

\begin{proof}
Note first that $ {Z_u} \triangleleft G $ and, therefore, if $ {Z_u} \leq \left\langle a \right\rangle $, then $ {Z_u} \triangleleft \left\langle a \right\rangle $ and the quotient group $ {{\left\langle a \right\rangle } \mathord{\left/
 {\vphantom {{\left\langle a \right\rangle } {{Z_u}}}} \right.
 \kern-\nulldelimiterspace} {{Z_u}}} $ exists. Letting $ {g_1} = {a^x}b $ and $ {g_2} = {a^y}b $, we have $ g_1^{ - 1}{g_2} = {({a^x}b)^{ - 1}}({a^y}b) = {b^{ - 1}}{a^{ - x + y}}b = {a^{y - x}} $. Thus, from Theorem \ref{t16} and Lemma \ref{c17}, $ \rho ({a^x}b) \circ {\left( {\rho (b)} \right)^{u - 1}} = \rho ({a^y}b) \circ {\left( {\rho (b)} \right)^{u - 1}} $ if and only if $ {a^{y - x}} \in {Z_u} $. This is equivalent to saying that $ {a^x}{Z_u} = {a^y}{Z_u} $ in the quotient group $ {{\left\langle a \right\rangle } \mathord{\left/
 {\vphantom {{\left\langle a \right\rangle } {{Z_u}}}} \right.
 \kern-\nulldelimiterspace} {{Z_u}}} $.
\end{proof}

We are now able to calculate the cardinality of the containers used in producing $ {\rm P}({D_m}) $ and $ \Lambda ({D_m}) $.

\begin{theorem}
\label{t19}
If $ u > 0 $ and $ {Z_u}({D_m}) \leq \left\langle a \right\rangle $, then 
\begin{enumerate}[(i)]
	\item $ \left| {\mathcal{C}({{( - 2)}^u},{{( - 2)}^{u - 1}})} \right| = \dfrac{m}{{\left| {{Z_u}({D_m})} \right|}} ,$ 
	\item $ \left| {\mathcal{C}({2^u},{2^{u - 1}})} \right| = \dfrac{m}{{\left| {{Z_u}({D_m})} \right|}} $.
\end{enumerate}
\end{theorem}

\begin{proof}
\begin{proof}[(i)]
By definition of container, 
\[\mathcal{C}({( - 2)^u},{( - 2)^{u - 1}}) = \left\{ {\mu ({{( - 2)}^u},x{{( - 2)}^{u - 1}}):x \in {\mathbb{Z}_m}} \right\} .
\]
If the domain of $x$ is $ {\mathbb{Z}_m} $, then the domain of $- x $ is also $ {\mathbb{Z}_m} $; thus, it follows by Lemma \ref{l15} that
\[
\left\{ {\mu ({{( - 2)}^u},x{{( - 2)}^{u - 1}}):x \in {\mathbb{Z}_m}} \right\} = \left\{ {\rho ({a^x}b) \circ \rho {{(b)}^{u - 1}}:x \in {\mathbb{Z}_m}} \right\} .
\]
Then it is clear that
\[
\left\{ {\rho ({a^x}b) \circ \rho {{(b)}^{u - 1}}:x \in {\mathbb{Z}_m}} \right\} = \left\{ {\rho ({a^x}b) \circ \rho {{(b)}^{u - 1}}:{a^x} \in \left\langle a \right\rangle } \right\} .
\]
Now, applying Corollary \ref{c18}, we see that the number of distinct mappings in this last set is the order of the quotient group $ \left| {{{\left\langle a \right\rangle } \mathord{\left/
 {\vphantom {{\left\langle a \right\rangle } {{Z_u}}}} \right.
 \kern-\nulldelimiterspace} {{Z_u}}}} \right| = \dfrac{m}{{\left| {{Z_u}} \right|}} $.
\noqed
\end{proof}
\begin{proof}[(ii)]
The argument proceeds as above noting that
\begin{align*}
\left\{ {\mu ({2^u},x{2^{u - 1}}):x \in {\mathbb{Z}_m}} \right\} 
	&= \left\{ {\lambda ({a^x}b) \circ \lambda {{(b)}^{u - 1}}:x \in {\mathbb{Z}_m}} \right\} \\
	&= \left\{ {{{( - 1)}^u}\left( {\rho ({a^x}b) \circ \rho {{(b)}^{u - 1}}} \right):x \in {\mathbb{Z}_m}} \right\} .
\end{align*}
Corollary \ref{c18} applies here, as it did in part \emph{(i)}, and the result follows in the same manner.
\end{proof} \noqed
\end{proof}

\begin{lemma}
\label{l20}
$ \left| {\mathcal{C}(0,1)} \right| = \dfrac{m}{{\left| {{Z_1}} \right|}} $.
\end{lemma}

\begin{proof}
Let $ x,y \in {\mathbb{Z}_m} $ and let $ \mu (0,x),\mu (0,y) $ be arbitrary elements of $ \mathcal{C}(0,1) $. If $ \mu (0,x) = \mu (0,y) $, then for every $ i \in {\mathbb{Z}_m},j \in {\mathbb{Z}_2} $, we have:
\[
({a^i}{b^j})\mu (0,x) = ({a^i}{b^j})\mu (0,y) .
\]
This is true if and only if, for each $j$, $ {a^{ - x{\beta _j}}} = {a^{ - y{\beta _j}}} $.And this is true if and only if $ x{\beta _j} \equiv y{\beta _j}\;(\bmod m) $, or, alternately, $ \left( {x - y} \right){\beta _j} \equiv 0\;(\bmod m) $.
If $ j = 0 $ then $ {\beta _0} = 0 $, so we need only consider the condition when $ j = 1 $ and, thus, $ {\beta _1} =  - 2 $. Thus we have $ \mu (0,x) = \mu (0,y) $ if and only if $- 2(x - y) \equiv 0\;(\bmod m) $. Since $- 1 $ is invertible in $ {\mathbb{Z}_m} $, this is equivalent to 
\begin{equation}
\label{e411}
2x \equiv 2y\;(\bmod m) .
\end{equation}
If $m$ is odd, $ 2 $ is invertible in $ {\mathbb{Z}_m} $, thus (\ref{e411}) is equivalent to $ x \equiv y\;(\bmod m) $ and, therefore, each mapping $ \mu (0,x) $ in $ \mathcal{C}(0,1) $ is distinct and $ \left| {\mathcal{C}(0,1)} \right| = m $. In this case the statement of the lemma is true since, by Theorem \ref{t14}, $ \left| {{Z_1}} \right| = 1 $. If $m$ is even, let us write $ m = {2^\ell }n $ with $ \ell  > 0 $ and $n$ odd. Condition (\ref{e411}) is then $ 2x \equiv 2y\;(\bmod {2^\ell }t) $. This is equivalent to $ x \equiv y\;(\bmod {2^{\ell  - 1}}t) $ by Lemma \ref{l10}. Therefore the elements of $ \mathcal{C}(0,1) $ are equal in pairs; $ \mu (0,x) = \mu (0,x + {2^{\ell  - 1}}n) $ for $ x = 0,1, \ldots ,{2^{\ell  - 1}}n - 1 $. Thus $ \left| {\mathcal{C}(0,1)} \right| = \dfrac{m}{2} $. By Theorem \ref{t14} we see that $ \left| {{Z_1}} \right| = 2 $ and the result is verified.
\end{proof}

\section{The Main Theorem}
We are now prepared to give formulas for the exact orders of $ {\rm P}({D_m}) $ and $ \Lambda ({D_m}) $.

\begin{theorem}
\label{t21}
If $ m = {2^\ell }n > 3 $ with $ n $ odd,
\begin{enumerate}[(i)]
	\item $  
	\left| {{\rm P}({D_m})} \right| = m\left( {\dfrac{1}{{\left| {{Z_1}} \right|}} + \sum\limits_{i = 1}^{t - 1} {\dfrac{1}{{\left| {{Z_i}} \right|}}} } \right)\mbox{, where }    
	t = \left\{ {
	\begin{array}{*{20}{c}}
 		{1 + or{d_m}( - 2){\text{ for }}\ell  = 0,\,n > 1} \\ 
  		{\ell  + pe{r_m}( - 2){\text{ for }}\ell  > 0,\,n > 1} \\ 
  		{\ell {\text{                   for }}\ell  > 0,\,n = 1} 
	\end{array}} \right.  ,$
	\item $ 
	 \left| {\Lambda ({D_m})} \right| = m\left( {\dfrac{1}{{\left| {{Z_1}} \right|}} + \sum\limits_{i = 1}^{t' - 1} {\dfrac{1}{{\left| {{Z_i}} \right|}}} } \right)\mbox{, where }  
	t' = \left\{ {
	 \begin{array}{*{20}{c}}
  		{1 + or{d_m}(2){\text{ for }}\ell  = 0,\,n > 1} \\ 
  		{\ell  + pe{r_m}(2){\text{ for }}\ell  > 0,\,n > 1} \\ 
 		{\ell {\text{                   for }}\ell  > 0,\, n= 1} 
	\end{array}} \right. , $
\end{enumerate}
\end{theorem}

\begin{proof}[Proof of (i)]
By Theorem \ref{t12}\emph{(i)} we have expressed $ {\rm P}({D_m}) $ as the disjoint union of containers
\begin{equation}
\label{e51}
\mathcal{C}(0,1) \cup \left( {\bigcup\limits_{i = 1}^t {\mathcal{C}({{( - 2)}^i},{{( - 2)}^{i - 1}})} } \right) ,
\end{equation}
where $ t = or{d_m}( - 2) $ for $ m $ odd, $ t = \ell  + pe{r_m}( - 2) - 1 $ for $ n > 1 $ and $ t = \ell  - 1 $ for $ n = 1 $. To find $ \left| {{\rm P}({D_m})} \right| $ we shall simply add the cardinalities of the containers in (\ref{e51}).  These cardinalities are given in Theorem \ref{t19}\emph{(i)} and Lemma \ref{l20} ; however, the hypothesis, $ {Z_u}({D_m}) \leq \left\langle a \right\rangle $, in Theorem \ref{t19} is not met in every case. Theorem \ref{t14}\emph{(ii)}, shows that the only case in which this hypothesis is not met is when $ n = 1 $ and $ u \geq \ell $. Thus consider the case in which $ m = {2^\ell } $ and $ u \geq \ell $. Here we see that we are taking the union as $i$ goes from $ 1 $ to $t$, but here $ t = \ell  - 1 < \ell  \leq u $. Thus we will not need to apply Theorem \ref{t19} in such a case and the hypothesis is irrelevant.
 
To complete the proof we apply Theorem \ref{t19}\emph{(i)} and Lemma \ref{l20} to the disjoint union (\ref{e51}) to obtain,
\begin{align*}
\left| {{\rm P}({D_m})} \right| &= \left| {\mathcal{C}(0,1)} \right| + \left( {\sum\limits_{i = 1}^t {\left| {\mathcal{C}({{( - 2)}^i},{{( - 2)}^{i - 1}})} \right|} } \right)\\ &= \dfrac{m}{{\left| {{Z_1}} \right|}} + \left( {\sum\limits_{i = 1}^t {\dfrac{m}{{\left| {{Z_i}} \right|}}} } \right) = m\left( {\dfrac{1}{{\left| {{Z_1}} \right|}} + \left( {\sum\limits_{i = 1}^t {\dfrac{1}{{\left| {{Z_i}} \right|}}} } \right)} \right).
\end{align*} 
The proof of \emph{(ii)} is quite similar for $ \Lambda ({D_m}) $. 
\end{proof}

At the end of the paper we have given a table displaying the orders of the commutation semigroups of dihedral groups $ {D_m} $ for $ 3 \leq m \leq 101 $. Looking at this table, many conjectures present themselves. For example, if $p$ is an odd prime less than 50, the table shows that $ \left| {{\rm P}({D_p})} \right| = \left| {{\rm P}({D_{2p}})} \right| $ and $ \left| {\Lambda ({D_p})} \right| = \left| {\Lambda ({D_{2p}})} \right| $. We will prove this true for all odd primes $p$. The proof will serve as an example of  how our approach can be applied to the study of such questions.

\begin{theorem}
\label{t22}
If $p$ is an odd prime, then 
\begin{enumerate}[(i)]
	\item $ \left| {{\rm P}({D_p})} \right| = \left| {{\rm P}({D_{2p}})} \right|, $
	\item $ \left| {\Lambda ({D_p})} \right| = \left| {\Lambda ({D_{2p}})} \right| $.
\end{enumerate}
\end{theorem}

\begin{proof}
The condition that $ \left| {{\rm P}({D_p})} \right| = \left| {{\rm P}({D_{2p}})} \right| $ can be written in terms of our formulas. First for $p$,
\[
\left| {{\rm P}({D_p})} \right| = p\left( {\dfrac{1}{{\left| {{Z_1}({D_p})} \right|}} + \sum\limits_{i = 1}^{t - 1} {\dfrac{1}{{\left| {{Z_i}({D_p})} \right|}}} } \right),
\] 
where $ t = 1 + or{d_p}( - 2) $. For $ 2p $ we have, 
\[
\left| {{\rm P}({D_{2p}})} \right| = 2p\left( {\dfrac{1}{{\left| {{Z_1}({D_{2p}})} \right|}} + \sum\limits_{i = 1}^{t' - 1} {\dfrac{1}{{\left| {{Z_i}({D_{2p}})} \right|}}} } \right), 
\] 
where $ t' = 1 + pe{r_{2p}}( - 2) $. By Theorem \ref{t14}, Each $ {Z_i}({D_p}) = \left\{ 1 \right\} $, thus we have
\[
\left| {{\rm P}({D_p})} \right| = p \cdot t = p \cdot (1 + or{d_p}( - 2)) .
\]
Looking at $ {D_{2p}} $, we see that since $ \ell  = 1 $ in this case, thus we are interested in $ {Z_u}({D_{2p}}) $ for $ u \geq \ell $. Theorem \ref{t14}\emph{(b)(i)} then shows that $ \left| {{Z_u}({D_{2p}})} \right| $ = $ {2^\ell }( = 2) $ for $ u \geq 1 $. Thus we have
 \[ \left| {{\rm P}({D_{2p}})} \right| = 2p\left( {\dfrac{1}{2} + \sum\limits_{i = 1}^{t' - 1} {\dfrac{1}{2}} } \right) = p \cdot t' = p \cdot (1 + pe{r_{2p}}( - 2)) .\]

To prove part \emph{(i)} of the theorem, it suffices to show that $ or{d_p}( - 2) = pe{r_{2p}}( - 2) $. By similar arguments, the same equation will imply that part \emph{(ii)} is true. 

Since $- 2 $ is coprime to $p$ it is invertible in $ {\mathbb{Z}_p} $, it has an order. Letting $ or{d_p}( - 2) = k $, we have $ {( - 2)^k} \equiv 1\;(\bmod p) $. Note that $ in{d_p}( - 2) = 1 $. By Lemma \ref{l10}, $ 2{( - 2)^k} \equiv 2\;(\bmod 2p) $, and multiplying both sides by $- 1 $ gives us, $ {( - 2)^{1 + k}} \equiv {( - 2)^1}\;(\bmod 2p) $. This shows that $ in{d_{2p}}( - 2) = 1 $. We see that the period of $- 2 $ is at most $k$, so suppose $ pe{r_{2p}}( - 2) = t < k $. Then $ {( - 2)^{1 + t}} \equiv {( - 2)^1}\;(\bmod 2p) $. Rewrite this as
 $- 2{( - 2)^t} \equiv  - 2\;(\bmod 2p) $ and multiply both sides by $- 1 $ to obtain, $ 2{( - 2)^t} \equiv 2\;(\bmod 2p) $. Applying Lemma \ref{l10} gives us $ {( - 2)^t} \equiv 1\;(\bmod p) $, contradicting  the minimality of $k$,  the order of $- 2 $ in $ {\mathbb{Z}_p} $. Therefore $ pe{r_{2p}}( - 2) = k = or{d_p}( - 2) $ and the result follows.
\end{proof}

In the Introduction we quoted Theorem C2, proved by Countryman in \cite{countryman70}. We can reprove this using our methods, but, more interestingly, we can prove a new result parallel to his. Our proof of Theorem C2 is similar to the proof below.

\begin{theorem}
\label{t23}
If $p$ and $q$ are primes, then $ \Lambda ({D_p}) \cong \Lambda ({D_q}) $ implies $ p = q $.
\end{theorem}

\begin{proof}
We will suppose that $ p < q $ and derive a contradiction. Note that if $ p = 2 $, $ {D_2} $ is abelian and, as mentioned earlier, $ \left| {\Lambda ({D_2})} \right| = 1 $. With $ q $ an odd prime, $ {D_q} $ is nonabelian, thus $ \left| {\Lambda ({D_q})} \right| > 1 $. The hypothesis is not met in such a case and, thus, we may assume that both $p$ and $q$ are odd primes. 

From our assumption, $ \Lambda ({D_p}) \cong \Lambda ({D_q}) $, it follows that $ \left| {\Lambda ({D_p})} \right| = \left| {\Lambda ({D_q})} \right| $. Since $p$ and $q$ are both odd, Theorem \ref{t14} implies that $ {Z_i} = \left\{ 1 \right\} $ in both cases. Thus Theorem \ref{t21} gives,
\begin{equation}
\label{e52}
p\left( {\dfrac{1}{{\left| {{Z_1}} \right|}} + \sum\limits_{i = 1}^{t - 1} {\dfrac{1}{{\left| {{Z_i}} \right|}}} } \right) = \left| {\Lambda ({D_p})} \right| = \left| {\Lambda ({D_q})} \right| = q\left( {\dfrac{1}{{\left| {{Z_1}} \right|}} + \sum\limits_{i = 1}^{t' - 1} {\dfrac{1}{{\left| {{Z_i}} \right|}}} } \right) ,
\end{equation}
where $ t = 1 + or{d_p}(2) $ and $ t' = 1 + or{d_q}(2) $. From (\ref{e52}) we deduce that $ p\left( t \right) = q\left( {t'} \right) $ and, hence, that
 $ p\left( {1 + or{d_p}(2)} \right) = q\left( {1 + or{d_q}(2)} \right) $. Since we are assuming that $p$ and $q$ are distinct primes we conclude that $ \left. q \right|1 + or{d_p}(2) $ and, therefore, $ q \leq 1 + or{d_p}(2) $. But $ or{d_p}(2) < p $, since $ \left. {or{d_p}(2)} \right|\phi (p) $, and $ \phi (p) = p - 1 $. Thus we have $ q \leq 1 + or{d_p}(2) < p $, which contradicts our assumption that $ p < q $.
\end{proof}

Note that the hypothesis, $ \Lambda ({D_p}) \cong \Lambda ({D_q}) $, can be ``weakened'' to $ \left| {\Lambda ({D_p})} \right| = \left| {\Lambda ({D_q})} \right| $. 

\section{Formulas and counterexamples}
First we will derive the formulas of Levy and DeWolf  quoted in the Introduction. Note that Theorem D covers all even values of $m$ including the nilpotent case ( $ m = {2^\ell } $ ) proved earlier by Levy in Theorem L2; thus, for brevity, we will derive only Theorems L1 and D. Translating from the notation of these theorems into the notation of this paper we have, $ in{d_R}(m) = pe{r_m}( - 2) $, $ in{d_L}(m) = pe{r_m}(2) $. Also, for $ n > 1 $, $ {m_R} = \ell  + pe{r_m}( - 2) $ and $ {m_L} = \ell  + pe{r_m}(2) $, and for $ n = 1 $, $ {m_R} = {m_L} = \ell . $ 
\begin{proof}[Derivation of Theorem L1]
First we will establish part \emph{(i)}. From Theorem \ref{t21} and Theorem \ref{t14}, for $m$ odd, we have $ \left| {{\rm P}({D_m})} \right| = m\left( {\dfrac{1}{{\left| {{Z_1}} \right|}} + \sum\limits_{i = 1}^{t - 1} {\dfrac{1}{{\left| {{Z_i}} \right|}}} } \right) $ with $ t = 1 + or{d_m}( - 2) $ and each $ {Z_i}({D_m}) = \left\{ 1 \right\} $. Thus 
  \[\left| {{\rm P}({D_m})} \right| = m\left( {1 + (t - 1)} \right) = mt = m(1 + or{d_m}( - 2)) = m(pe{r_m}( - 2) + 1) .\]
The proof for part \emph{(ii)} is similar.
\end{proof}
\begin{proof}[Derivation of Theorem D]
We first prove part \emph{(i)}. Letting $ m = {2^\ell }s $, with $ s{\text{ odd}} $, and $ \ell ,s \in {\mathbb{Z}^ + } $, we will break into two cases: $ s > 1 $ and $ s = 1 $. Supposing that $ s > 1 $, Theorem \ref{t14} implies that $ \left| {{Z_i}} \right| = {2^i} $ for $ i < \ell $ and $ \left| {{Z_i}} \right| = {2^\ell } $ for $ i \geq \ell $. From Theorem \ref{t21} we see that
\begin{equation}
\label{e62}
\left| {{\rm P}({D_m})} \right| = m\left( {\dfrac{1}{{\left| {{Z_1}} \right|}} + \sum\limits_{i = 1}^{t - 1} {\dfrac{1}{{\left| {{Z_i}} \right|}}} } \right),
\end{equation}
with $ t = \ell  + pe{r_m}( - 2) $. Note that
\begin{align}
\label{e63}
\sum\limits_{i = 1}^{t - 1} {\dfrac{1}{{\left| {{Z_i}} \right|}}}  &= \sum\limits_{i = 1}^{\ell  - 1} {\dfrac{1}{{{2^i}}}}  + \sum\limits_{i = \ell }^{t - 1} {\dfrac{1}{{{2^\ell }}}}  = \dfrac{{{2^{\ell  - 1}} - 1}}{{{2^{\ell  - 1}}}} + \dfrac{{t - \ell }}{{{2^\ell }}} \notag\\
&= \dfrac{{{2^\ell } - 2 + t - \ell }}{{{2^\ell }}} = \dfrac{{{2^\ell } - 2 + pe{r_m}( - 2)}}{{{2^\ell }}}.
\end{align}
Substituting the outcome of (\ref{e63}) into (\ref{e62}) we arrive at,
\begin{align*}
  \left| {{\rm P}({D_m})} \right| &= m\left( {\dfrac{1}{{\left| {{Z_1}} \right|}} + \sum\limits_{i = 1}^{t - 1} {\dfrac{1}{{\left| {{Z_i}} \right|}}} } \right) = m\left( {\dfrac{1}{2} + \dfrac{{{2^\ell } - 2 + pe{r_m}( - 2)}}{{{2^\ell }}}} \right) \\
			  	 &= {2^\ell }s\left( {\dfrac{{{2^\ell } + {2^{\ell  - 1}} - 2 + pe{r_m}( - 2)}}{{{2^\ell }}}} \right) = s\left( {{2^\ell } + {2^{\ell  - 1}} + pe{r_m}( - 2) - 2} \right) .
\end{align*}
The proof of part \emph{(ii)} is similar.
\end{proof}
Having shifted among notations to come up with these derivations, we would like to give here the simplest numerical formulas we know for $ {\rm P}({D_m})$ and $\Lambda ({D_m}) $.
\begin{enumerate}[(i)]
	\item For $ m $ odd, \[ \left| {{\rm P}({D_m})} \right| = m(or{d_m}( - 2) + 1) \mbox{ and } \left| {\Lambda ({D_m})} \right| = m(or{d_m}(2) + 1), \]
	\item For $ m = {2^\ell }n $ with $ n $ odd and $ n > 1 ,$ 
		 \[ \left| {{\rm P}({D_m})} \right| = s({2^\ell } + {2^{\ell  - 1}} - 2 + pe{r_m}( - 2))  \mbox{ and } \left| {\Lambda ({D_m})} \right| = s({2^\ell } + {2^{\ell  - 1}} - 2 + pe{r_m}(2)) .\]
	\item For $ m = {2^\ell } $, $ \left| {{\rm P}({D_m})} \right| = \left| {\Lambda ({D_m})} \right| = {2^\ell } + {2^{\ell  - 1}} - 2 $.
\end{enumerate}
Lastly, we will comment on the three theorems of Countryman on $pq$ groups quoted in the Introduction.

We show here that $ \left| {{\rm P}({D_8})} \right| = \left| {\Lambda ({D_8})} \right| $ and $ {\rm P}({D_8}) \cong \Lambda ({D_8}) $, but $ {\rm P}({D_8})\not  = \Lambda ({D_8}) $, thereby giving counterexamples to the statements $ (c) \Rightarrow (a){\text{ and }}(b) \Rightarrow (a) $ in Theorem C1 for dihedral groups in general. Applying Theorem \ref{t12} to the dihedral group of order 16,we see that
\[
{\rm P}({D_8}) = \mathcal{C}(0,1)\,\dot  \cup\, \mathcal{C}(6,1)\,\dot  \cup\, \mathcal{C}(4,2) \mbox{ while } \Lambda ({D_8}) = \mathcal{C}(0,1)\,\dot  \cup \,\mathcal{C}(2,1)\,\dot  \cup\, \mathcal{C}(4,2).
\]
By Lemma \ref{l20}, $ \left| {\mathcal{C}(0,1)} \right| = \dfrac{8}{{\left| {{Z_1}} \right|}} = \dfrac{8}{2} = 4 $, $ \left| {\mathcal{C}(6,1)} \right| = \dfrac{8}{{\left| {{Z_1}} \right|}} = \dfrac{8}{2} = 4 $, $ \left| {\mathcal{C}(2,1)} \right| = \dfrac{8}{{\left| {{Z_1}} \right|}} = \dfrac{8}{2} = 4 $, and
 $ \left| {\mathcal{C}(4,2)} \right| = \dfrac{8}{{\left| {{Z_2}} \right|}} = \dfrac{8}{4} = 2 $.Therefore $ \left| {{\rm P}({D_8})} \right| = 10 = \left| {\Lambda ({D_8})} \right| $, however, by Lemma 6, $ \mathcal{C}(6,1) \cap \mathcal{C}(2,1) = \varnothing $ and, therefore, $ {\rm P}({D_8})\not  = \Lambda ({D_8}) $. To show that $ {\rm P}({D_8}) \cong \Lambda ({D_8}) $, we define a mapping $ \phi :{\rm P}({D_8}) \to \Lambda ({D_8}) $ 
by $ \phi \left( {\mu (x,y)} \right) = \mu (3x,y) $. This can be shown to be an isomorphism. Note also that the container decomposition above shows that $ {\rm P}({D_8})\not  \subseteq \Lambda ({D_8}) $ and $ {\rm P}({D_8})\not  \subseteq \Lambda ({D_8}) $, thus providing a counterexample to the generalization of Theorem C3 to dihedral groups in general. 

We will show that $ (c) \Rightarrow (b) $ fails for $ {D_{15}} $. Note by our Theorem 21 it follows that $ \left| {{\rm P}({D_{15}})} \right| = 75 = \left| {\Lambda ({D_{15}})} \right| $. To see that $ {\rm P}({D_{15}})\not  \cong \Lambda ({D_{15}}) $ we shall apply the following theorem: 
\begin{atheorem}{1 (Gupta, \cite{gupta66})}
$ {\rm P}({D_m}) \cong \Lambda ({D_m}) $ if and only if $ or{d_p}(2) \equiv 0\;(\bmod 4) $ for every odd prime factor of $m$.  
\end{atheorem}
Since $ or{d_3}(2) = 2\not  \equiv 0\;(\bmod 4) $, it follows that $ {\rm P}({D_{15}})\not  \cong \Lambda ({D_{15}}) $, and we have our counterexample.
 
It is also the case that $ {\rm P}({D_{10}}) \cong {\rm P}({D_5}) $ and $ \Lambda ({D_{10}}) \cong \Lambda ({D_5}) $. This gives a counterexample to the generalization of Theorem C2 and our Theorem \ref{t23}. The container decomposition has proven quite useful in finding these isomorphisms.  

\begin{center}

\begin{tabular}{||c || c | c || c || c | c || c || c | c |}
	\hline
	 $ m $ & $ \left| \Rho \left( D_{ m } \right)  \right| $ & $ \left| \Lambda \left( D_{ m } \right)  \right| $ & $ m $ & $ \left| \Rho \left( D_{ m } \right)  \right| $ & $ \left| \Lambda \left( D_{ m } \right)  \right| $ & $ m $ & $ \left| \Rho \left( D_{ m } \right)  \right| $ & $ \left| \Lambda \left( D_{ m } \right)  \right| $ \\ \hline
	3 &	6 &	9 &	36 &	63 &	90 &	69 &	1587 &	1587 \\ \hline 
4 &	4 &	4 &	37 &	1369 &	1369 &	70 &	455 &	455 \\ \hline 
5 &	25 &	25 &	38 &	190 &	361 &	71 &	5041 &	2556 \\ \hline 
6 &	6 &	9 &	39 &	507 &	507 &	72 &	117 &	144 \\ \hline 
7 &	49 &	28 &	40 &	70 &	70 &	73 &	1387 &	730 \\ \hline 
8 &	10 &	10 &	41 &	861 &	861 &	74 &	1369 &	1369 \\ \hline 
9 &	36 &	63 &	42 &	147 &	147 &	75 &	1575 &	1575 \\ \hline 
10 &	25 &	25 &	43 &	344 &	645 &	76 &	247 &	418 \\ \hline 
11 &	66 &	121 &	44 &	99 &	154 &	77 &	2387 &	2387 \\ \hline 
12 &	15 &	18 &	45 &	585 &	585 &	78 &	507 &	507 \\ \hline 
13 &	169 &	169 &	46 &	529 &	276 &	79 &	6241 &	3160 \\ \hline 
14 &	49 &	28 &	47 &	2209 &	1128 &	80 &	130 &	130 \\ \hline 
15 &	75 &	75 &	48 &	69 &	72 &	81 &	2268 &	4455 \\ \hline 
16 &	22 &	22 &	49 &	2107 &	1078 &	82 &	861 &	861 \\ \hline 
17 &	153 &	153 &	50 &	525 &	525 &	83 &	3486 &	6889 \\ \hline 
18 &	36 &	63 &	51 &	459 &	459 &	84 &	210 &	210 \\ \hline 
19 &	190 &	361 &	52 &	208 &	208 &	85 &	765 &	765 \\ \hline 
20 &	40 &	40 &	53 &	2809 &	2809 &	86 &	344 &	645 \\ \hline 
21 &	147 &	147 &	54 &	270 &	513 &	87 &	2523 &	2523 \\ \hline 
22 &	66 &	121 &	55 &	1155 &	1155 &	88 &	165 &	220 \\ \hline 
23 &	529 &	276 &	56 &	112 &	91 &	89 &	2047 &	1068 \\ \hline 
24 &	33 &	36 &	57 &	570 &	1083 &	90 &	585 &	585 \\ \hline 
25 &	525 &	525 &	58 &	841 &	841 &	91 &	1183 &	1183 \\ \hline 
26 &	169 &	169 &	59 &	1770 &	3481 &	92 &	598 &	345 \\ \hline 
27 &	270 &	513 &	60 &	120 &	120 &	93 &	1023 &	1023 \\ \hline 
28 &	70 &	49 &	61 &	3721 &	3721 &	94 &	2209 &	1128 \\ \hline 
29 &	841 &	841 &	62 &	341 &	186 &	95 &	3515 &	3515 \\ \hline 
30 &	75 &	75 &	63 &	441 &	441 &	96 &	141 &	144 \\ \hline 
31 &	341 &	186 &	64 &	94 &	94 &	97 &	4753 &	4753 \\ \hline 
32 &	46 &	46 &	65 &	845 &	845 &	98 &	2107 &	1078 \\ \hline 
33 &	198 &	363 &	66 &	198 &	363 &	99 &	1584 &	3069 \\ \hline 
34 &	153 &	153 &	67 &	2278 &	4489 &	100 &	600 &	600 \\ \hline 
35 &	455 &	455 &	68 &	204 &	204 &	101 &	10201 &	10201 \\ \hline 
\end{tabular}
\end{center}

\begin{samepage}

\end{samepage}

\end{document}